\documentclass[12pt]{amsart}
\usepackage{pgf,tikz,pgfplots} 
\pgfplotsset{compat=1.14}
\usepackage{mathrsfs}
\usetikzlibrary{arrows}
\usetikzlibrary{patterns}
\usepackage{pgfplots}
\usetikzlibrary{intersections, pgfplots.fillbetween}
\usepackage[colorlinks=true,urlcolor=blue, citecolor=red,linkcolor=blue,linktocpage,pdfpagelabels, bookmarksnumbered,bookmarksopen]{hyperref}
\usepackage[hyperpageref]{backref}
\usepackage{cleveref}
\usepackage{xcolor}
\usepackage{amsthm} 
\usepackage{latexsym,amsmath,amssymb}
\usepackage{accents}
\usepackage[colorinlistoftodos,prependcaption,textsize=tiny]{todonotes}
\usepackage{a4wide}
\usepackage{soul}
\usepackage{mathtools} 
\usepackage{xparse} 

\pgfdeclarelayer{ft}
\pgfdeclarelayer{bg}
\pgfsetlayers{bg,main,ft}

\title[Obstacles for Sobolev-homeomorphisms with low rank]{Obstacles for Sobolev-homeomorphisms with low rank\\-- pointwise a.e. vs distributional Jacobians --}

\author{Woongbae Park}
\address[Woongbae Park]{Department of Mathematics,
215 Carnegie Building, Syracuse University, Syracuse, NY 13244}
\email{wpark06@syr.edu}
\author{Armin Schikorra}
\address[Armin Schikorra]{Department of Mathematics,
University of Pittsburgh,
301 Thackeray Hall,
Pittsburgh, PA 15260, USA}
\email{armin@pitt.edu}

\definecolor{indigo}{rgb}{0.29, 0.0, 0.51}
\definecolor{p1}{gray}{0.4}
\definecolor{p2}{gray}{0.6}
\definecolor{p3}{gray}{0.98}
\definecolor{p4}{gray}{0.8}
\definecolor{p5}{gray}{0.9}

plus 6pt minus 12pt
\def\eps{\varepsilon}

\def\B{\mathbb{B}}

\def\S{{\mathbb S}}

\newtheorem{theorem}{Theorem}
\newtheorem{lemma}[theorem]{Lemma}
\newtheorem{corollary}[theorem]{Corollary}

\newtheorem{definition}[theorem]{Definition}

\newtheorem{question}[theorem]{Question}

\def\rank{{\rm rank\,}}

\newcommand{\R}{\mathbb{R}}

\newcommand{\brac}[1]{\left (#1 \right )}
\newcommand{\abs}[1]{\left |#1 \right |}
\newcommand{\Ep}{\bigwedge\nolimits}

\newcommand{\barint}{
\rule[.036in]{.12in}{.009in}\kern-.16in \displaystyle\int }

\newcommand{\barcal}{\mbox{$ \rule[.036in]{.11in}{.007in}\kern-.128in\int $}}

\def\mvint_#1{\mathchoice
          {\mathop{\vrule width 6pt height 3 pt depth -2.5pt
                  \kern -8pt \intop}\nolimits_{\kern -3pt #1}}%
          {\mathop{\vrule width 5pt height 3 pt depth -2.6pt
                  \kern -6pt \intop}\nolimits_{#1}}%
          {\mathop{\vrule width 5pt height 3 pt depth -2.6pt
                  \kern -6pt \intop}\nolimits_{#1}}%
          {\mathop{\vrule width 5pt height 3 pt depth -2.6pt
                  \kern -6pt \intop}\nolimits_{#1}}}

\numberwithin{theorem}{section} \numberwithin{equation}{section}

\newcommand{\aleq}{\precsim}

\let\latexchi\chi
\makeatletter
\renewcommand\chi{\@ifnextchar_\sub@chi\latexchi}
\newcommand{\sub@chi}[2]{
  \@ifnextchar^{\subsup@chi{#2}}{\latexchi_{#2}}%
}
\newcommand{\subsup@chi}[3]{
  \latexchi_{#1}^{#3}%
}
\makeatother

\makeatletter
\def\tikz@arc@opt[#1]{
  {%
    \tikzset{every arc/.try,#1}%
    \pgfkeysgetvalue{/tikz/start angle}\tikz@s
    \pgfkeysgetvalue{/tikz/end angle}\tikz@e
    \pgfkeysgetvalue{/tikz/delta angle}\tikz@d
    \ifx\tikz@s\pgfutil@empty%
      \pgfmathsetmacro\tikz@s{\tikz@e-\tikz@d}
    \else
      \ifx\tikz@e\pgfutil@empty%
        \pgfmathsetmacro\tikz@e{\tikz@s+\tikz@d}
      \fi%
    \fi
    \tikz@arc@moveto
    \xdef\pgf@marshal{\noexpand%
    \tikz@do@arc{\tikz@s}{\tikz@e}
      {\pgfkeysvalueof{/tikz/x radius}}
      {\pgfkeysvalueof{/tikz/y radius}}}%
  }%
  \pgf@marshal%
  \tikz@arcfinal%
}
\let\tikz@arc@moveto\relax
\def\tikz@arc@movetolineto#1{%
  \def\tikz@arc@moveto{\tikz@@@parse@polar{\tikz@arc@@movetolineto#1}(\tikz@s:\pgfkeysvalueof{/tikz/x radius} and \pgfkeysvalueof{/tikz/y radius})}}
\def\tikz@arc@@movetolineto#1#2{#1{\pgfpointadd{#2}{\tikz@last@position@saved}}}
\tikzset{%
  move to start/.code=\tikz@arc@movetolineto\pgfpathmoveto,%
  line to start/.code=\tikz@arc@movetolineto\pgfpathlineto}
\makeatother

\begin{document}
\begin{abstract}
We show that for any $k$ and $s > \frac{k+1}{k+2}$ there exist neither $W^{s,\frac{k}{s}}$-Sobolev nor $C^s$-H\"older homeomorphisms from the disk $\B^n$ into $\R^N$ whose gradient has rank $< k$ \emph{in distributional sense}. This complements known examples of such kind of homeomorphisms whose gradient has rank $<k$ \emph{almost everywhere}.
\end{abstract}
\maketitle

\section{Introduction}
Throughout this paper let $n \geq 2$ and $s \in (0,1]$. In \cite{FMCO18,LM16} it was shown that for any $s\in (0,1)$ and $k \in \{2,\ldots,n\}$ there exists a $C^s$-homeomorphism (onto its target) $u: \B^n \to \R^n$ such that $\nabla u \in L^1$ and
\begin{equation}\label{eq:ranknablau}
\rank(\nabla u)< k \quad \text{a.e. in $\B^n$}.
\end{equation}

As is well known, the pointwise a.e. derivative is a way less restrictive object than the distributional derivative which captures more fine geometric properties. The simplest example is the Heaviside function which has a.e. vanishing derivative, but is certainly not constant -- because the \emph{distributional} derivative does not vanish. Effects of this type are also known for distributional vs. a.e. Jacobians, see for example \cite{MR4030269,LS21}.

The purpose of this note is to show that results similar to \cite{FMCO18,LM16} are wrong if the pointwise a.e. notion in \eqref{eq:ranknablau} is replaced with a distributional version.

Since a map $u \in C^s$ does not need to be differentiable, the notion of distributional $rank \nabla u$ might not be immediate, but the idea is simple: Take any $k$-form monomial in $\Ep^k \R^N$
\[
 dp^{i_1} \wedge \ldots \wedge dp^{i_k}.
\]
For any smooth map $f: \B^n \to \R^N$, the pullback
\[
 f^\ast(dp^{i_1} \wedge \ldots dp^{i_k}) = \sum_{\alpha_1,\ldots,\alpha_k=1}^n\partial_{\alpha_1} f^{i_1}\, \partial_{\alpha_2} f^{i_2}\ldots \partial_{\alpha_k} f^{i_k}\ dx^{\alpha_1} \wedge \ldots dx^{\alpha_k }
\]
is a $k$-form whose components are the determinants of $k \times k$-submatrices of $(\nabla f^{i_1},\ldots,\nabla f^{i_k}) \in \R^{n\times k}$. In particular, if $f$ is differentiable then $\rank Df < k$ is equivalent to
\[
 f^\ast(dp^{i_1} \wedge \ldots dp^{i_k}) = 0 \quad \forall 1 \leq i_1 < \ldots< i_k \leq N.
\]
On the other hand, $k \times k$-determinants of submatrices of $(\nabla f^{i_1},\ldots,\nabla f^{i_k}) \in \R^{k \times N}$ are Jacobians and those can be defined in a distributional sense for H\"older and Sobolev maps.

Consequently, it is reasonable to say $f: \B^n \to \R^N$ and $k \in \{1,\ldots,N\}$ satisfies
\[
\rank Df < k \quad \text{in distributional sense}
\]
if for any $1 \leq i_1 < \ldots< i_k \leq N$ we have
\[
 f^\ast(dp^{i_1} \wedge \ldots dp^{i_k}) = 0
\]
in distributional sense -- we shall recall the precise meaning of the latter in \Cref{s:distributionalrank}.

Here is our main result:
\begin{theorem}\label{th:nohomok}
Fix any $s \geq \frac{k}{k+1}$, $k \geq 2$. Denote by $\B^n$ the unit ball in $\R^n$. There exists no homeomorphism $u \in W^{s,\frac{k}{s}}(\B^n,\R^N)$ such that
\[
 \rank (\nabla u) < k \quad \text{in distributional sense}.
\]
\end{theorem}
Here and henceforth for $s \in (0,1)$ the fractional Sobolev space $W^{s,p}(\Omega)$ is the one induced by the Gagliardo seminorm
\[
 [f]_{W^{s,p}(\Omega)} := \brac{\int_{\Omega} \int_{\Omega} \frac{|f(x)-f(y)|^{p}}{|x-y|^{n+sp}}\, dx\, dy}^{\frac{1}{p}}.
\]

By slicing arguments and Fubini's theorem, cf. \Cref{la:slicing}, any $W^{s,\frac{k}{s}}(\B^n,\R^N)$-homeomorphism for $k < n$ induces a $W^{s,\frac{k}{s}}(\B^k,\R^N)$-homeomorphism on the $k$-dimensional ball $\B^k$ -- and the above notion of distributional rank is stable under that slicing operation. Thus, when proving \Cref{th:nohomok} one can assume w.l.o.g. $k=n$. Actually, in the case $k=n$, we don't even have to assume that $u: \B^n \to \R^N$ is a homeomorphism onto its target, it is only used that $u$ restricted to the boundary $\partial \B^n$ is one-to-one. Precisely we have
\begin{theorem}\label{th:nohomon}
Let
\[
 f: \S^{n-1} \to \R^N
\]
be a homeomorphism.

Then for any $s\in (0,1]$, $s \geq \frac{n}{n+1}$, there exists no map $u \in W^{s,\frac{n}{s}}(\B^n,\R^N)$ (homeomorphism or not) with the properties
\begin{itemize}
 \item $\rank Du < n$ in distributional sense in $\B^n$
 \item $f=u\Big |_{\partial \B^n}$ in the sense of traces
\end{itemize}
\end{theorem}

Since by Sobolev embedding $C^{s+\eps} \hookrightarrow W^{s,\frac{n}{s}}_{loc}$, we have in particular

\begin{corollary}\label{co:nohomeoCs}
Let $s > \frac{k}{k+1}$. The $C^s$-homeomorphism as the one constructed in \cite{FMCO18,LM16} can not exist if the assumption
\[
 \rank Du < k \quad \text{a.e.}
\]
is substituted with
\[
 \rank Du < k \quad \text{in distributional sense}.
\]
\end{corollary}

The latter strengthens in particular \cite[Theorem 12]{FMCO18} where it is shown that a homeomorphism such as the one constructed in \cite{FMCO18} cannot exist if we additionally assume it belongs to $W^{1,k}$. Observe that for $W^{1,k}$-maps distributional and a.e. notion of $\rank Du < k$ coincide.

The assumption $s > \frac{k}{k+1}$ in \Cref{co:nohomeoCs} is notable. The notion of $\rank Df < k$ in distributional sense is well-defined for $C^s$-maps with $s > 1-\frac{1}{k}$, see \Cref{s:distributionalrank}.
\begin{question}\label{q:q}
Does \Cref{co:nohomeoCs} hold for $s \in (1-\frac{1}{k}, \frac{k}{k+1})$?
\end{question}

Indeed, \Cref{q:q} is related to a conjecture by Gromov that one cannot $C^s$-embedd two-dimensional surfaces into the Heisenberg group $\mathbb{H}_1$ for $s > \frac{1}{2}$. This is known to be true for $s >\frac{2}{3}$, \cite{Gromov96}. Notably Wenger and Young \cite{WengerYoung18} recently constructed examples of $C^s$-maps from two-dimensional surfaces into the Heisenberg group, whose \emph{boundary map} is homeomorphic -- which suggests that the threshhold $s \geq \frac{n}{n+1}$ in \Cref{th:nohomon} could be sharp at least in some dimensions. As a matter of fact, our main tool is a consequence of a technique developed for maps into the Heisenberg group in \cite{HS23}: Any homeomorphism can be nontrivially ``linked'' with a differential form, more precisely

\begin{lemma}[\cite{HS23}]
Let $f: \S^{n-1} \to \R^N$, $N \geq n+2$, be a smooth homeomorphism. There exist $\omega \in C_c^\infty(\Ep^{n-1}\R^N)$ such that $d\omega \equiv 0$ in a neighborhood of $f(\S^{n-1}) \subset \R^N$ and \begin{equation} \label{eq:linknonzero}\int_{\S^{n-1}} f^\ast(\omega) = 1.\end{equation}
\end{lemma}
See \Cref{la:HSomega} for more explanation and a slightly sharper statement. The main idea to prove \Cref{th:nohomon} is to show that the rank condition is not compatible with \eqref{eq:linknonzero}.

Lastly, let us mention that \Cref{co:nohomeoCs} can also be proven for the limiting H\"older case $C^{\frac{k}{k+1}_+}$, the adaptations are left to the reader -- the underlying technical arguments are discussed in \cite{HS23}.

{\bf Acknowledgement:}
Funding is acknowledged as follows

\begin{itemize}
\item A.S. is an Alexander-von-Humboldt Fellow.
\item A.S. is funded by NSF Career DMS-2044898.
\end{itemize}

\section{The distributional rank condition}\label{s:distributionalrank}
The theory of distributional Jacobians has a long tradition with celebrated contributions by Ball \cite{B76}, Brezis, Nirenberg \cite{BN95}, Coifman, Lions \cite{CLMS}, M\"uller \cite{M90}, Reshetnyak \cite{R68}, Wente \cite{W69}, Tartar \cite{T79}, among many others. We recall and adapt the notion to our setting, but the results of this section are likely well-known at least to experts.

\begin{definition}\label{def:distrank}
Let $s \in (0,1]$, $p \in (1,\infty)$. We say that $f \in W^{s,p}(\B^n,\R^N)$ satisfies
\[
 f^\ast(dp^{i_1} \wedge \ldots dp^{i_k}) = 0
\]
in ($W^{s,p}$-)distributional sense if for any $\varphi \in C_c^\infty(\B^n)$ and any $f_\eps \in C^\infty(\overline{\B^n},\R^N)$ with
\[
 \|f_\eps-f\|_{W^{s,p}} \xrightarrow{\eps \to 0} 0
\]
we have
\[
 \lim_{\eps \to 0} \int_{\B^n} f_\eps^\ast(dp^{i_1} \wedge \ldots dp^{i_k}) \varphi = 0.
\]
\end{definition}

The following type of estimate is essentially known since \cite{SY99}, inspired by the results in \cite{CLMS}. In special cases an extremely elegant proof using harmonic extensions was given in \cite{BN2011}, see also \cite{LS20} for the relation to commutator estimates and harmonic extensions. See also \cite{SVS20} where this is revisited using the arguments of \cite{LS20}.
\begin{lemma}\label{s:distrankest}
Let $s>1-\frac{1}{k}$ and $p > k$.

For any $f \in W^{s,p}(\B^n,\R^N)$ and $f_\eps$ as in \Cref{def:distrank}
\[
 f^\ast(dp^{i_1} \wedge \ldots dp^{i_k})[\varphi]\equiv f^\ast_{W^{s,p}}(dp^{i_1} \wedge \ldots dp^{i_k})[\varphi] := \lim_{\eps \to 0} \int_{\B^n} f_\eps^\ast(dp^{i_1} \wedge \ldots dp^{i_k}) \varphi = 0.
\]
is a linear functional on $\varphi \in C_c^\infty(\B^n)$, independent of the precise choice of $f_\eps$.

Moreover we have
\[
 \abs{f^\ast(dp^{i_1} \wedge \ldots dp^{i_k})[\varphi]} \aleq [f]_{W^{s,p}}^{k} [\varphi]_{W^{(1-s)k,\frac{p}{p-k}}}
\]
and for $f_1,f_2 \in W^{s,p}(\B^n,\R^N)$ we have
\[
 \abs{f_1^\ast(dp^{i_1} \wedge \ldots dp^{i_k})[\varphi]-f_2^\ast(dp^{i_1} \wedge \ldots dp^{i_k})[\varphi]} \aleq [f_1-f_2]_{W^{s,p}(\B^n)}\, \brac{[f_1]_{W^{s,p}(\B^n)}^{k-1} + [f_2]^{k-1}_{W^{s,p}(\B^n)}} [\varphi]_{W^{(1-s)k,\frac{p}{p-k}}}
\]
In particular if $f \in W^{s_1,p_1}(\B^n,\R^N) \cap W^{s_2,p_2}(\B^n,\R^N)$ then the linear functional as an $W^{s_1,p_1}$-limit or an $W^{s_2,p_2}$-limit coincides, i.e.
\[
f_{W^{s_1,p_1}}^\ast(dp^{i_1} \wedge \ldots dp^{i_k})[\varphi]=f_{W^{s_2,p_2}}^\ast(dp^{i_1} \wedge \ldots dp^{i_k})[\varphi]
\]
for all $\varphi \in C_c^\infty(\B^n,\R^N)$.
\end{lemma}

Observe that this gives naturally also a suitable notion of $C^s$-distributional rank, because $C^s$ embedds into $W^{s-\eps,q}$ for any $\eps > 0$, $q \in (1,\infty)$.

\Cref{th:nohomok} is a consequence of \Cref{th:nohomon}, once we observe the following restriction result.
\begin{lemma}\label{la:slicing}
Assume that $f \in W^{s,p}(\S^{k+1},\R^N)$ be continuous  for $s>\frac{k+1}{k+2}$, $p >k$. Then, if we slice
\[
 \S^{k+1} = \bigcup_{t \in [-1,1]} \{t\} \times \sqrt{1-t^2} \S^{k}
\]
Then for $\mathcal{L}^1$-a.e. $t \in (-1,1)$
\begin{equation}\label{eq:sliceasdas}
 g:= f \Big |_{\{t\} \times \sqrt{1-t^2} \S^{k}} \in W^{s,p}(\{t\} \times \sqrt{1-t^2} \S^{k},\R^N)
\end{equation}
and if for some $j \in \{1,\ldots,{k+1}\}$ we have
\[
 \rank Df < j \quad \text{in distributional sense in $\S^{k+1}$}
\]
then
\[
 \rank Dg < j \quad \text{in distributional sense in $\S^k$}
\]
\end{lemma}
\begin{proof}
The fact that \eqref{eq:sliceasdas} holds is an application of Fubini's theorem, see e.g. \cite[Lemma A.2]{LS21}. For a more direct proof see \cite[\textsection 6.2.]{Leoni}.

Similarly, the rank condition also follows from Fubini's theorem, since
\[
 0 = \lim_{\eps \to 0} \int_{\S^{k+1}} f_\eps^\ast(dp^{i_1} \wedge \ldots dp^{i_j}) \varphi =
  \lim_{\eps \to 0} \int_{(-1,1)} \int_{\{t\} \times \sqrt{1-t^2} \S^{k}} f_\eps^\ast(dp^{i_1} \wedge \ldots dp^{i_j}) \varphi
\]
implies that for a.e. $t \in (-1,1)$,
\[
\lim_{\eps \to 0} \int_{\{t\} \times \sqrt{1-t^2} \S^{k}} f_\eps^\ast(dp^{i_1} \wedge \ldots dp^{i_j}) \varphi =0.
\]
\end{proof}

\section{Using the linking number: Proof of Theorem~\ref{th:nohomon}}\
As discussed in the introduction, the proof of \Cref{th:nohomon} is based on the recent arguments developed for maps into the Heisenberg group \cite{S20,HS23,HMS29}.

The main ingredient is the following lemma, which is essentially just a reformulation of the well-known fact from Algebraic Topology that the homology class $H_{N-(n-1)-1}(\R^N \setminus K))$ is nontrivial if $K$ is a homeomorphic to the $\S^{n-1}$-sphere. Indeed this fact can be found in the early chapters of any algebraic topology book, see e.g. \cite[Corollary 1.29]{Vick-1994}. However, this particular reformulation transforms this fact into an analytically easily usable tool.

\begin{lemma}\label{la:HSomega}
Let $f: \S^{n-1} \to \R^N$, $N \geq n+2$, be a homeomorphism. There exist $\eps > 0$ (depending on $f$) and $\omega \in C_c^\infty(\Ep^{n-1}\R^N)$ such that
\begin{itemize}
 \item $d\omega \equiv 0$ in a neighborhood of $f(\S^{n-1}) \subset \R^N$
 \item for any $\tilde{f} \in C^\infty(\S^{n-1},\R^N)$ with
 \[
  \|\tilde{f}-f\|_{L^\infty(\S^{n-1})} < \eps
 \]
we have
\[
 \int_{\S^{n-1}} \tilde{f}^\ast(\omega) = 1.
\]
\end{itemize}
\end{lemma}
For a proof we refer to \cite[Proposition 9.2.]{S20} or \cite[Lemma 2.6.]{HS23}. $d\omega$ essentially represents a ``surface'' that is linked with $f(\S^{n-1})$ -- where ``surface'' is to be understood in a general sense, see \cite{H19}.

\begin{proof}[Proof of \Cref{th:nohomon}]
Assume to the contrary the existence of $u$ as in \Cref{th:nohomon}.

By \cite[Lemma A.4]{LS21} there exist $u_\eps \in C^\infty(\B^n)$ such that $[u_\eps-u]_{W^{s,\frac{n}{s}}(\B^n)} \xrightarrow{\eps \to 0} 0$ and $u_\eps \Big |_{\partial \B^n}$ uniformly converges to $f$ as $\eps \to 0$.

From \Cref{la:HSomega} we find $\omega$ so that $d\omega \equiv 0$ around $f(\S^{n-1})$ and
\[
 \int_{\S^{n-1}} u_\eps^\ast(\omega) = 1\quad  \forall \eps \ll 1
\]
By Stokes' theorem this is equivalent to
\begin{equation}\label{eq:appllemma}
 \int_{\B^n} u_\eps^\ast(d\omega) = 1 \quad \forall \eps \ll 1
\end{equation}
If we write
\[
d\omega = \sum_{I} \kappa_I dp^I
\]
where each $I=(i_1,\ldots,i_{n})$ is a strictly ordered tuple in $\{1,\ldots,N\}$, then we have by \eqref{eq:appllemma}
\[
1= \sum_{I} \int_{\B^n} u_\eps^\ast (dp^I) \kappa_I(u_\eps) \quad \forall \eps \ll 1
\]
Moreover, by the uniform Lipschitz continuity of $\kappa$
\begin{equation}\label{eq:convkappaiueps}
 \|\kappa_I(u_\eps)  -\kappa_I(u)\|_{W^{s,\frac{n}{s}}(\B^n)} \aleq \|u_\eps -u\|_{W^{s,\frac{n}{s}}(\B^n)}  \xrightarrow{\eps \to 0} 0.
\end{equation}
Since $d\omega \equiv 0$ in a neighborhood of $f(\S^{n-1})$ we also have $d\omega \equiv 0$ in a neighborhood of $u_\eps(\S^{n-1})$ for all suitably small $\eps$. For such $\eps$ we have $\kappa_I(u_\eps) \equiv 0$ on $\partial \B^n$ (in the classical sense).

Since the trace operator is an extension of the continuous trace operator, we find that
\[
 \kappa_I(u_\eps) \in W^{s,\frac{n}{s}}_0(\B^n).
\]
By the convergence \eqref{eq:convkappaiueps} we conclude
\[
 \kappa_I(u) \in W^{s,\frac{n}{s}}_0(\B^n).
\]
Here $W^{s,p}_0(\B^n)$ the closure of $C_c^\infty(\B^n)$-maps under the $W^{s,p}(\B^n)$-norm and we have used the trace identification with the continuous trace (since $s\frac{n}{s}=n>1$), see \cite[Chapter 9.]{Leoni}.

Thus we find
\[
\varphi_\delta \in C_c^\infty(\B^n)
\]
with
\[
[\varphi_\delta-\kappa_I(u)]_{W^{s,\frac{n}{s}}(\B^n)} \xrightarrow{\delta \to 0} 0.
\]
In particular for $\eps$ and $\delta$ suitably small we have
\[
[\varphi_\delta-\kappa_I(u_\eps)]_{W^{s,\frac{n}{s}}(\B^n)} \ll 1 \quad \forall \eps \in (0,\eps_0), \delta \in (0,\delta_0).
\]
Recall that $I$ is an $n$-tuple. By the continuity of distributional Jacobian, \Cref{s:distrankest} for $k=n$, since $s\geq \frac{n}{n+1} > 1-\frac{1}{n}$ and $\frac{n}{s} > n$, we have
\[
 \sum_{I} \int_{\B^n} u_\eps^\ast (dp^I) \brac{\varphi_\delta - \kappa_I(u_\eps)} \aleq [u]_{W^{s,\frac{n}{s}}(\B^n)}^{n} [ \varphi_\delta - \kappa_I(u_\eps)]_{W^{(1-s)n,\frac{1}{1-s}}}
\]
Since $s \geq \frac{n}{n+1}$ we see that $(1-s)n \leq s$, and thus by Sobolev embedding
\[
 \sum_{I} \int_{\B^n} u_\eps^\ast (dp^I) \brac{\varphi_\delta - \kappa_I(u_\eps)} \aleq [u]_{W^{s,\frac{n}{s}}(\B^n)}^{n} [ \varphi_\delta - \kappa_I(u_\eps)]_{W^{s,\frac{n}{s}}}.
\]
We conclude that for $\eps_1$ and $\delta_1$ suitably small,
\[
\frac{1}{2} \leq \sum_{I} \int_{\B^n} u_\eps^\ast (dp^I) \varphi_\delta \quad \forall \eps \in (0,\eps_1), \delta \in (0,\delta_1)
\]
On the other hand, by the assumption of $\rank Du \leq n-1$ in distributional sense, we have
\[
 \lim_{\eps \to 0} \int_{\B^n} u_\eps^\ast (dp^I) \varphi_\delta  =0.
\]
This is a contradiction, and we can conclude.
\end{proof}

\bibliographystyle{abbrv}
\bibliography{bib}

\end{document}